\documentclass[oneside,english,reqno]{amsart}
\usepackage{mathpazo}
\usepackage[T1]{fontenc}
\usepackage[latin9]{inputenc}
\usepackage{color}
\usepackage{babel}
\usepackage{amsthm}
\usepackage{amssymb}
\usepackage[pdfborder={0 0 0},colorlinks=true,pdfborderstyle={},linkcolor=black,citecolor=black,urlcolor=blue]{hyperref}
\usepackage{breakurl}

\makeatletter
  \theoremstyle{plain}
  \newtheorem*{thm*}{Theorem}

\usepackage{calrsfs}
\usepackage{graphicx}
\usepackage{color}
\usepackage{perpage}

\MakePerPage{footnote}
\gdef\SetFigFontNFSS#1#2#3#4#5{} 

\theoremstyle{remark}
\newtheorem*{qst*}{Question}
\begin{document}

\title[Mirror model]
{Lower bound for the escape probability in the Lorentz Mirror Model on $\mathbb Z^2$}

\author{Gady Kozma}
\address{Gady Kozma,
Department of Mathematics, Ziskind Building, 
Weizmann Institute of Science, Rehovot 76100, Israel}
\email{gady.kozma@weizmann.ac.il}
\author{Vladas Sidoravicius}
\address{Vladas Sidoravicius, IMPA, Estrada Dona Castorina 110, 
22460-320, Rio de Janeiro, Brasil.}
\email{vladas@impa.br}

\begin{abstract}
We show that in the Lorentz mirror model, at any density of mirrors, 
\[
\mathbb{P}(0\leftrightarrow\partial Q(n))\ge\frac{1}{2n+1}.
\]\end{abstract}
\maketitle

\bigskip
\bigskip

Let $0 \le p \le 1$. Designate each vertex $x \in \mathbb Z^2$ a double-sided 
mirror with probability $p$. For every vertex the designations are independent 
and identically distributed. Vertices which are designated a  mirror, with probability 
$1/2$ obtain a north-west mirror, otherwise they obtain north-east mirror. A ray 
of light traveling along the edges of $\mathbb Z^2$ is reflected when
it hits a mirror (see image on the right)
and keeps its direction unchanged at vertices which are not designated a mirror. The question 
is if for all values of $0 < p \leq 1$ the orbits of the ray of light are periodic, or, otherwise, for 
some values of $p < 1$ there is a positive probability that the light travels to infinity. 
For $p=1$, a simple argument, see for instance \cite[\S 13.3]{G}, shows that this question is 
equivalent to the question of the existence of an infinite open path in the 
bond percolation model on $\mathbb{Z}^2$ at the parameter value $1/2$, 
which is, due to the seminal result of \cite{H}, known to have negative answer. If $p=1$ and 
mirrors are oriented with probabilities $p_{_{NW}} \neq p_{_{NE}}$ we
have the same conclusion, see  \cite[p.\ 54]{K}.
No similar result is known for $p < 1$. We are ready to state
our theorem. By $[0 \leftrightarrow A]$ we denote the event that ray of light
starting at the origin reaches set $A \subset \mathbb{Z}^2$, and  let
$Q(n) = [-n, n]^2$. 
\parshape 13 0pt 10cm 0pt 10cm 0pt 10cm 0pt 10cm 0pt 10cm 0pt 10cm 0pt 10cm 0pt 10cm 0pt 10cm 0pt 10cm 0pt 10cm 0pt 10cm 0pt \hsize
\vadjust{\kern -7.6cm \vtop to 7.6cm{\hbox to \hsize{\hss\includegraphics{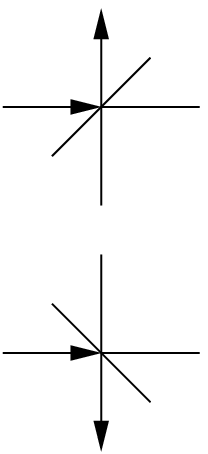}}}}

\begin{thm*}
In the mirror model at any density $0 < p \leq 1$ of mirrors, 
\[
\mathbb{P}_p(0\leftrightarrow Q(n)^c)\ge\frac{1}{2n+1}.
\]
\end{thm*}
\begin{proof}
Examine the mirror model on an infinite cylinder $\mathbb Z \times S_{2n+1}$ of odd width $2n+1$. We first note
that, deterministically, there must exist an infinite path. Indeed,
examine the paths crossing the $2n+1$ horizontal edges whose left-end vertex has the coordinate $0$.
 Then each finite path must cross an even number of edges,
since each crossing moves it from the left half of the cylinder or back.
Since the total number of edges is odd, at least one cannot belong
to a finite path, hence it belongs to an infinite path.

Hence the expected number of edges that belong to an infinite path
is $\ge1$. Since the cylinder is invariant under rotations, we get that
the probability that it crosses any given vertex is equal to $\frac{1}{2n+1}$
of this expectation. So it is $\ge\frac{1}{2n+1}$.

Finally, since the path cannot tell the difference between being in
the cylinder and in $\mathbb{Z}^{2}$ before getting to distance $n$,
we get the desired claim on $\mathbb{Z}^{2}$.
\end{proof}
Remarks
\begin{enumerate}
\item The argument works for reasonable high-dimensional analogs of the
mirror model.
\item The argument does not apply to the periodic Manhattan model (see
  \cite[p.\ 13]{C}) because the Manhattan
model has the evenness built into it, and cannot be defined on a cylinder
with odd width consistently. Thus we have avoided a contradiction as
the result is not true for the the Manhattan model. 
It is not difficult to convince oneself that the path of the ray is
contained inside the vacant $*$-cluster of the origin. Therefore 
when the density of obstacles is bigger than the critical value
for site percolation on $\mathbb{Z}^2$, one has that
$\mathbb{P}_p(0\leftrightarrow Q(n)^c)$ decays exponentially.
\item On the other hand, the argument does apply to the \emph{randomly
  oriented} Manhattan model. In this model the orientations of the
  ``streets'' and ``avenues'' are random and i.i.d\@. Here there is no
  problem to define the model on a cylinder with odd width and the
  proof carries through literally.
\item Rotating mirrors. In this model mirrors are changing their position deterministically after each
interaction with the beam of light by flipping by $90^{\rm{o}}$ degrees. If $p=1$, it is easy to see that
the path of the ray of light is unbounded. However it is not known to the authors if in this case it is recurrent.
\end{enumerate}

We thank S. Smirnov for a useful discussion.

\bigskip






\end{document}